\documentclass[11pt]{amsart}
\usepackage[latin1]{inputenc}
\usepackage[T1]{fontenc}
\usepackage{hyperref}
\usepackage{amsfonts}
\usepackage{amsmath}
\usepackage{epsf, subfigure, verbatim}
\usepackage{amssymb}
\usepackage{amsthm}
\usepackage{latexsym}
\usepackage{layout}
\usepackage{amsrefs}

\newtheorem{theorem}{Theorem}

\newtheorem{lemma}{Lemma}[section]
\newtheorem{corollary}{Corollary}

\newtheorem{remark}{Remark}[section]

\newcommand{\ee}{\mathrm{e}}

\newcommand{\reals}{\mathbb{R}}
\newcommand{\ind}{\mathbf{1}}
\newcommand{\e}{\mathbb{E}}
\newcommand{\p}{\mathbb{P}}
\newcommand{\drift}{\mathtt c}
\newcommand{\wq}{w^{(q)}}
\newcommand{\hq}{h^{(q)}}
\newcommand{\zq}{z^{(q)}}

\begin{document}

\title[]{On the time spent in the red by a refracted L\'evy risk process}

\author[J.-F. Renaud]{Jean-Fran\c{c}ois Renaud}
\address{D\'epartement de math\'ematiques, Universit\'e du Qu\'ebec \`a Montr\'eal (UQAM), 201 av.\ Pr\'esident-Kennedy, Montr\'eal (Qu\'ebec) H2X 3Y7, Canada}
\email{renaud.jf@uqam.ca}

\date{\today}

\keywords{Spectrally negative L\'{e}vy processes, refraction, occupation times, bankruptcy, Parisian ruin.}

\begin{abstract}
In this paper, we introduce an insurance ruin model with adaptive premium rate, thereafter refered to as restructuring/refraction, in which classical ruin and bankruptcy are distinguished. In this model, the premium rate is increased as soon as the wealth process falls into the \textit{red zone} and is brought back to its regular level when the process recovers. The analysis is mainly focused on the time a refracted L\'evy risk process spends in the red zone (analogous to the duration of the negative surplus). Building on results from \cite{kyprianouloeffen2010} and \cite{loeffenetal2012}, we identify the distribution of various functionals related to occupation times of refracted spectrally negative L\'evy processes. For example, these results are used to compute the probability of bankruptcy and the probability of Parisian ruin in this model with restructuring.
\end{abstract}

\maketitle

\section{Introduction}

In classical actuarial ruin theory, the time of default is assumed to occur if and when the surplus process falls below a certain threshold level for the first time. Without loss of generality, which is due to the spatial homogeneity of most surplus processes, this threshold level has commonly been assumed to be the \textit{artificial} level $0$. For solvency purposes, it is more appropriate to view this threshold level as the insurer's solvency capital requirement (SCR) set by the regulatory body. Therefore, new risk concepts and models have recently been introduced: Parisian ruin (see, e.g., \cite{czarnapalmowski2010}, \cite{loeffenetal2011}, \cite{landriaultetal2013}), random observations (see, e.g., \cite{albrecheretal2011a}) and Omega models (see, e.g., \cite{albrecheretal2011b}, \cite{gerberetal2012} and \cite{albrecherlautscham2013}).

Our goal is to introduce an insurance ruin model where default and bankruptcy are disentangled, and where also \textit{restructuring} is considered. Indeed, it seems very likely that when the company is in financial distress, namely when the surplus process falls below the critical level (SCR), some sort of restructuring will be undertaken. We propose a model with adaptive premium, i.e., where the premium are increased as soon as the surplus process is in the so-called \textit{red zone} and are brought back to their regular level when things get better; it is assumed that this critical situation is due to temporarily bad luck. Therefore, to do so, we will use a refracted L\'evy risk process as our surplus process, as studied by Kyprianou and Loeffen \cite{kyprianouloeffen2010}. 

In conclusion, we propose to study different occupation-time related definitions of bankruptcy/default in a L\'evy risk model with restructuring (a refracted L\'evy risk model). Despite the generality of using a L\'evy process as the underlying surplus process, our model is very tractable, thanks to the work of Kyprianou and Loeffen \cite{kyprianouloeffen2010} and the fluctuation identities they have obtained. This type of risk processes has been used traditionally to build models with a constant threshold dividend strategy; see the references in \cite{kyprianouloeffen2010} and \cite{kyprianouetal2012a}. Recently, in \cite{kyprianouetal2012}, a number of identities concerned with the distribution of occupation times until first passage times for a refracted L\'evy process were obtained. Instead of borrowing results directly from \cite{kyprianouetal2012}, we will use results and techniques from \cite{loeffenetal2012} to derive (more general) results expressed solely in terms of the scale functions of the underlying process.

As a consequence, our new identities for the distribution of occupation times of refracted L\'evy processes could form the theoretical basis to further develop a set of risk measures in this L\'evy risk model with restructuring/refraction.

\subsection{The model}

Let $U$ be the surplus process of interest. We choose the level $b > 0$ to be the threshold level representing the insurer's solvency capital requirement. As soon as $U$ goes below this critical level, restructuring will be undertaken in the sense that the premium for \textit{large} claims will be (temporarily) increased. When $U$ recovers, that is when $U$ goes above $b$ again, then things come back as they were initially. In other words, we consider the level $b$ to be the critical level and the interval $(-\infty,b)$ the \textit{red zone}; when the surplus process $U$ falls below $b$, a restructuring of the business is undertaken and materializes itself by an increase in the \textit{drift} of the process.

Intuitively, we are interested in the following dynamic:
$$
\mathrm{d}U_t = \mathrm{d}Y_t + \alpha \ind_{\{U_t < b\}} \mathrm{d}t ,
$$
or, equivalently,
$$
U_t = Y_t + \alpha \int_0^t \ind_{\{U_s < b\}} \mathrm{d}s ,
$$
where $Y$ is the underlying (uncontrolled) risk process during \textit{standard} business periods. However, for ease of presentation, namely to make our paper closer in notation to \cite{kyprianouloeffen2010}, we will instead use the following equivalent point of view: we first set the dynamic of the surplus process in the \textit{red zone} and refract it when it goes above $b$. Mathematically, let $X$ be the risk process during periods of financial distress, and define $U$ as follows: for $\alpha < \e [X_1]$,
$$
\mathrm{d}U_t = \mathrm{d}X_t - \alpha \ind_{\{U_t > b\}} \mathrm{d}t .
$$

For example, if $X$ is of the form
$$
X_t = \drift t - S_t ,
$$
where $\drift>\alpha$ represents the premium rate and where the driftless subordinator $S = (S_t)_{t \geq 0}$ represents the aggregate claim payments, then $U$ has a drift value of $\drift$ in the \textit{red zone} (below $b$) and a drift value of $\drift-\alpha$ above $b$. This includes the Cram\'er-Lundberg risk process as a special case. However, in what follows, $X$ will be a general spectrally negative L\'evy risk process.

Our main result gives representations for the joint Laplace transforms of
$$
\left( \kappa_a^- , \int_0^{\kappa_a^-} \ind_{\{U_s < b\}} \mathrm{d}s \right) \quad \text{and} \quad \left( \kappa_c^+ , \int_0^{\kappa_c^+} \ind_{\{U_s < b\}} \mathrm{d}s \right) ,
$$
where $a \leq x,b \leq c$, and where $\kappa_a^-$ and $\kappa_c^+$ are first passage times, in terms of the so-called scale functions of the underlying L\'evy processes $X$ and $Y$. These quantities will then be used to study the probability of bankruptcy and the probability of Parisian ruin for $U$.

The rest of the paper is organized as follows. Next, we introduce spectrally negative L\'evy processess and refracted L\'evy processes, including useful identities involving scale functions. Section 3 presents the main result of the paper, while in Section 4 and Section 5 corollaries are derived and applied to the computations of the probability of bankruptcy and the probability of Parisian ruin respectively.

\section{Spectrally negative L\'evy processes}

On the filtered probability space $(\Omega, (\mathcal F_t)_{t\geq0},\mathbb P)$, let $X = (X_t)_{t \geq 0}$ be a spectrally negative L\'evy process (SNLP), that is a  process with stationary and independent increments and no positive jumps. Hereby we exclude the case that $X$ is the negative of a subordinator, i.e., we exclude the case of $X$ having decreasing paths. The law of $X$ such that $X_0 = x$ is denoted by $\p_x$ and the corresponding expectation by $\e_x$. We write $\p$ and $\e$ when $x=0$. As the L\'{e}vy process $X$ has no positive jumps, its Laplace transform exists: for $\lambda,t \geq 0$,
$$
\e \left[ \mathrm{e}^{\lambda X_t} \right] = \mathrm{e}^{t \psi(\lambda)} ,
$$
where
$$
\psi(\lambda) = \gamma \lambda + \frac{1}{2} \sigma^2 \lambda^2 + \int^{\infty}_0 \left( \mathrm{e}^{-\lambda z} - 1 + \lambda z \ind_{(0,1]}(z) \right) \Pi(\mathrm{d}z) ,
$$
for $\gamma \in \reals$ and $\sigma \geq 0$, and where $\Pi$ is a $\sigma$-finite measure on $(0,\infty)$ such that
$$
\int^{\infty}_0 (1 \wedge z^2) \Pi(\mathrm{d}z) < \infty .
$$
This measure is called the L\'{e}vy measure of $X$, while $(\gamma,\sigma,\Pi)$ is refered to as the L\'evy triplet of $X$. Note that for convenience we define the L\'evy measure in such a way that it is a measure on the positive half line instead of the negative half line. Further, note that $\e \left[ X_1 \right] = \psi'(0+)$.

The process $X$ has paths of bounded variation if and only if $\sigma=0$ and $\int^{1}_0 z \Pi(\mathrm{d}z)<\infty$. In that case we denote by
$\drift:=\gamma+\int^{1}_0 z \Pi(\mathrm{d}z) > 0$ the so-called drift of $X$ which can now be written as
$$
X_t = \drift t - S_t ,
$$
where $S = (S_t)_{t \geq 0}$ is a driftless subordinator (for example a Gamma process or a compound Poisson process with positive jumps). When $S$ is a compound Poisson process and $\e \left[ X_1 \right] = \psi'(0+) > 0$, we recover the classical Cram\'er-Lundberg risk process. Finally, if $\Pi(\mathrm{d}z) \equiv 0$, we recover the Brownian motion risk process, i.e., $X$ can then be written as
$$
X_t = \drift t + \sigma B_t ,
$$
since $\drift = \gamma$ and where $B = (B_t)_{t \geq 0}$ is a standard Brownian motion. In the actuarial risk theory literature, SNLPs have been called general L\'evy insurance risk processes.

\subsection{Scale functions and fluctuation identities}

For an arbitrary SNLP, the Laplace exponent $\psi$ is strictly convex and $\lim_{\lambda \to \infty} \psi(\lambda) = \infty$. Thus, there exists a function $\Phi \colon [0,\infty) \to [0,\infty)$ defined by $\Phi(q) = \sup \{ \lambda \geq 0 \mid \psi(\lambda) = q\}$ (its right-inverse) such that
$$
\psi ( \Phi(q) ) = q, \quad q \geq 0 .
$$
We have that $\Phi(q)=0$ if and only if $q=0$ and $\psi'(0+)\geq0$.

We now recall the definition of the $q$-scale function $W^{(q)}$. For $q \geq 0$, the $q$-scale function of the process $X$ is defined as the continuous function with Laplace transform
\begin{equation*}
\int_0^{\infty} \mathrm{e}^{- \lambda y} W^{(q)} (y) \mathrm{d}y = \frac{1}{\psi(\lambda) - q} , \quad \text{for $\lambda > \Phi(q)$.}
\end{equation*}
This function is unique, positive and strictly increasing for $x\geq0$ and is further continuous for $q\geq0$. We extend $W^{(q)}$ to the whole real line by setting $W^{(q)}(x)=0$ for $x<0$.  We write $W = W^{(0)}$ when $q=0$. The initial value of $W^{(q)}$ is known to be
\begin{equation*}
W^{(q)}(0)=
\begin{cases}
1/\drift & \text{when $\sigma=0$ and $\int_{0}^1 z \Pi(\mathrm{d}z) < \infty$},  \\
0 & \text{otherwise},
\end{cases}
\end{equation*}
where we used the following definition: $W^{(q)}(0) = \lim_{x \downarrow 0} W^{(q)}(x)$. We will also frequently use the following function
$$
Z^{(q)}(x) = 1 + q \int_0^x{W}^{(q)}(y)\mathrm dy, \quad x\in\mathbb R.
$$

Fix $\alpha > 0$ and define $Y = (Y_t)_{t \geq 0}$ by $Y_t = X_t - \alpha t$. In what follows, if $X$ has paths of bounded variation, then it is assumed that
\begin{equation}\label{E:delta}
0 < \alpha < \drift = \gamma + \int_{(0,1)} z \Pi(\mathrm{d}z) .
\end{equation}
This condition is very intuitive. Indeed, recall that when the process $X$ has paths of bounded variation, then we can write $X_t = \drift t - S_t$, where $S = (S_t)_{t \geq 0}$ is a driftless subordinator. Condition~\eqref{E:delta} says that we do not want to remove all the drift. Clearly, $Y$ is also a spectrally negative L\'evy process; in fact, it has the same Gaussian coefficient $\sigma$ and L\'evy measure $\Pi$ as $X$. Its Laplace exponent is given by $\lambda \mapsto \psi(\lambda) - \alpha \lambda$, with right-inverse
$$
\varphi (q) = \sup \left\lbrace \lambda \geq 0 \colon \psi(\lambda) - \alpha \lambda = q \right\rbrace .
$$
The law of $Y$ such that $Y_0 = y$ is denoted by $\mathbf{P}_y$ and the corresponding expectation by $\mathbf{E}_y$. For each $q \geq 0$, we will write $\mathbb{W}^{(q)}$ and $\mathbb{Z}^{(q)}$ for the scale functions associated with $Y$.

Now, for any $a,c \in \reals$, define the following stopping times
\begin{gather*}
\tau_a^- = \inf \{t > 0 \colon X_t < a \} \quad \text{and} \quad \tau_c^+ = \inf \{t > 0 \colon X_t > c \} ,\\
\nu_a^- = \inf \{t > 0 \colon Y_t < a \} \quad \text{and} \quad \nu_c^+ = \inf \{t > 0 \colon Y_t > c \} ,
\end{gather*}
with the convention $\inf\emptyset=\infty$.

It is well known that, if $a \leq x \leq c$, then the solution to the two-sided exit problem for $X$ is given by
\begin{equation}\label{E:exitabove}
\e_x \left[ \ee^{-q \tau_c^+} ; \tau_c^+ < \tau_a^- \right] = \frac{W^{(q)}(x-a)}{W^{(q)}(c-a)} ,
\end{equation}
\begin{equation}\label{E:exitbelow}
\e_x \left[ \ee^{-q \tau_a^-} ; \tau_a^- < \tau_c^+ \right] = Z^{(q)}(x-a) - \frac{Z^{(q)}(c-a)}{W^{(q)}(c-a)} W^{(q)}(x-a) ,
\end{equation}
where, for a random variable $Z$ and an event $A$, $\e [Z;A] := \e [Z \ind_A]$.
Finally, in general, the \textit{classical} probability of ruin is given by
\begin{equation}\label{E:classicalruinproba}
\p_x \left( \tau_0^- < \infty \right) = 1 - \left( \e \left[ X_1 \right] \vee 0 \right) W(x) ,
\end{equation}
where $\e \left[ X_1 \right] = \psi'(0+)$.

Of course, the same results hold for $Y$; for example,
\begin{equation*}
\mathbf{E}_x \left[ \ee^{-q \nu_c^+} ; \nu_c^+ < \nu_a^- \right] = \frac{\mathbb{W}^{(q)}(x-a)}{\mathbb{W}^{(q)}(c-a)} .
\end{equation*}

\subsection{Refracted L\'evy processes}

Fix $b > 0$ and consider the following stochastic differential equation:
\begin{equation}\label{E:dynamic}
\mathrm{d}U_t = \mathrm{d}X_t - \alpha \ind_{\{U_t > b\}} \mathrm{d}t , \quad t \geq 0 .
\end{equation}

\begin{theorem}[Kyprianou and Loeffen \cite{kyprianouloeffen2010}]
For a fixed $X_0=x \in \mathbb{R}$, there exists a unique strong solution $U = (U_t)_{t \geq 0}$ to Equation~\eqref{E:dynamic}. Moreover, $U$ is a strong Markov process.
\end{theorem}

We now present fluctuation identities for refracted processes. First, we define the following functions related to $U$ (and to the scale functions of the underlying SNLPs $X$ and $Y$): for $x,a \in \reals$ and $q \geq 0$, define
\begin{align*}
\wq (x;a) &= W^{(q)} (x-a) + \alpha \ind_{\{x \geq b\}} \int_b^x \mathbb{W}^{(q)}(x-y) W^{(q) \prime}(y-a) \mathrm{d}y ,\\
\zq (x;a) &= Z^{(q)} (x-a) + \alpha q \ind_{\{x \geq b\}} \int_b^x \mathbb{W}^{(q)}(x-y) W^{(q)}(y-a) \mathrm{d}y .
\end{align*}
Note that when $x < b$,
$$
\wq (x;a) = W^{(q)} (x-a) , \quad \zq (x;a) = Z^{(q)} (x-a) , \quad \hq (x) = \mathrm{e}^{\Phi(q)x}  .
$$
Most of the notation follows \cite{kyprianou2012}.

As we will now see, one could think of these functions as being the scale functions of the refracted process $U$. First, for any $a,c \in \reals$, define the following stopping times
$$
\kappa_a^- = \inf \{t > 0 \colon U_t < a \} \quad \text{and} \quad \kappa_c^+ = \inf \{t > 0 \colon U_t > c \} .
$$
The next result provides a solution to the two-sided exit problem for $U$. It is essentially a re-statement of Theorem 4 in \cite{kyprianouloeffen2010} (see also \cite{kyprianou2012}) and it generalizes the expressions in Equations~\eqref{E:exitabove} and~\eqref{E:exitbelow} corresponding to the case $\alpha=0$.
\begin{theorem}[Kyprianou and Loeffen \cite{kyprianouloeffen2010}]\label{T:twosidedforU}
For $q \geq 0$ and $a \leq x,b \leq c$ we have
$$
\e_x \left[ \mathrm{e}^{-q \kappa_c^+} ; \kappa_c^+ < \kappa_a^- \right] = \frac{\wq(x;a)}{\wq(c;a)} ,
$$
and
$$
\e_x \left[ \mathrm{e}^{-q \kappa_a^-} ; \kappa_a^- < \kappa_c^+ \right] = \zq(x;a) - \frac{\zq(c;a)}{\wq(c;a)} \wq(x;a) .
$$
\end{theorem}
\begin{proof}
The case when $a=0$ has been proved in \cite{kyprianouloeffen2010}. For $a \in \mathbb{R}$ such that $a \leq x,b \leq c$, we use a quasi-space-homogeneity property of $U$:
$$
\e_x \left[ \mathrm{e}^{-q \kappa_c^+} ; \kappa_c^+ < \kappa_a^- \right] = \e_{x-a} \left[ \mathrm{e}^{-q \tilde{\kappa}_{c-a}^+} ; \tilde{\kappa}_{c-a}^+ < \tilde{\kappa}_0^- \right] ,
$$
where $\tilde{\kappa}_{c-a}^+$ and $\tilde{\kappa}_0^- $ represent stopping times associated with the solution of
\begin{equation*}
U_t = X_t - \alpha \int_0^t \ind_{\{U_s > b-a\}} \mathrm{d}s , \quad t \geq 0 .
\end{equation*}
Using Theorem 4 in \cite{kyprianouloeffen2010} and changing variables, we get
\begin{multline*}
\e_{x-a} \left[ \mathrm{e}^{-q \tilde{\kappa}_{c-a}^+} ; \tilde{\kappa}_{c-a}^+ < \tilde{\kappa}_0^- \right] = W^{(q)} (x-a) + \alpha \ind_{\{x-a \geq b-a\}} \int_{b-a}^{x-a} \mathbb{W}^{(q)}(x-a-y) W^{(q) \prime}(y) \mathrm{d}y \\
= W^{(q)} (x-a) + \alpha \ind_{\{x \geq b\}} \int_b^x \mathbb{W}^{(q)}(x-y) W^{(q) \prime}(y-a) \mathrm{d}y
\end{multline*}
and the result follows. The second identity is derived in the same way.
\end{proof}

Finally, the probability of \textit{classical} ruin for a refracted L\'evy process is given by
$$
\p_x \left( \kappa_0^- < \infty \right) = 1 - \left( \frac{\mathbf{E} \left[ Y_1 \right] \vee 0}{1-\alpha W(b)} \right) w^{(0)}(x;0) ,
$$
where $\mathbf{E} \left[ Y_1 \right] = \e \left[ X_1 \right] - \alpha$. In our case, since it is assumed that $\alpha < \e [X_1]$ (net profit condition), we have
\begin{multline}\label{E:refractedruinproba}
\p_x \left( \kappa_0^- < \infty \right) = 1 - \left( \frac{\e \left[ X_1 \right] - \alpha}{1-\alpha W(b)} \right) w^{(0)}(x;0) \\
= 1 - \left( \frac{\e \left[ X_1 \right] - \alpha}{1-\alpha W(b)} \right) \left\lbrace W(x) + \alpha \ind_{\{x \geq b\}} \int_b^x \mathbb{W}(x-y) W^\prime(y) \mathrm{d}y \right\rbrace .
\end{multline}

\begin{remark}
If no restructuring is undertaken, i.e., if $\alpha=0$, then $U=X=Y$ and there is only one process in the model. Then, the probability of ruin is the one given in~\eqref{E:classicalruinproba}.
\end{remark}

\begin{remark}
It is important to note that $U$, and therefore $\{w^{(q)}, q \geq 0\}$ and $\{z^{(q)}, q \geq 0\}$, all depend on the fixed values of $\alpha$ and $b$.
\end{remark}

Before stating the main results of this paper, we present a few identities relating the different scale functions. We can show (by taking Laplace transforms on both sides of the equation) that, for $p,q,x \geq 0$,
\begin{multline*}
\alpha \int_0^x \mathbb{W}^{(p)}(x-y) W^{(q)}(y) \mathrm{d}y + (p-q) \int_0^x \int_0^y \mathbb{W}^{(p)}(y-z) W^{(q)}(z) \mathrm{d}z \mathrm{d}y \\
= \int_0^x \mathbb{W}^{(p)}(y) \mathrm{d}y - \int_0^x W^{(q)}(y) \mathrm{d}y ,
\end{multline*}
which is a generalization of the first displayed equation in Section 8 of \cite{kyprianouloeffen2010}. Differentiating with respect to $x$ yields the following identity
\begin{multline}\label{E:convolution}
(q-p) \int_0^x \mathbb{W}^{(p)}(x-y) W^{(q)}(y) \mathrm{d}y \\
= W^{(q)}(x) - \mathbb{W}^{(p)}(x) + \alpha \left( W^{(q)}(0) \mathbb{W}^{(p)}(x) + \int_0^x \mathbb{W}^{(p)}(x-y) W^{(q) \prime}(y) \mathrm{d}y \right) ,
\end{multline}
which is a generalization of Equation~(5) in \cite{loeffenetal2012}. Further, we derive, for $x>b$,
\begin{equation}\label{E:rep_wq}
\wq (x;0) = \left( 1-\alpha W^{(q)}(0) \right) \mathbb{W}^{(q)} (x) - \alpha \int_0^b \mathbb{W}^{(q)}(x-y) W^{(q) \prime}(y) \mathrm{d}y .
\end{equation}

\subsection{An example}

There are various examples of SNLPs for which an explicit formula exists for the scale function $W^{(q)}$. For example, when $X$ is a compound Poisson process risk process with a jump distribution that has a Laplace transform which is the ratio of two polynomials, then the Laplace transform of the scale function $1/(\psi(\lambda)-q)$ is also a rational function and an explicit expression for the scale function $W^{(q)}$ is known.

We now present the case of a L\'evy jump-diffusion process where the jump distribution is a mixture of exponentials. In other words,
$$
X_t = \drift t + \sigma B_t - \sum_{i=1}^{N_t} \xi_i ,
$$
where $\sigma>0$, $\drift \in \mathbb R$, $B=(B_t)_{t \geq 0}$ is a Brownian motion, $N=(N_t)_{t \geq 0}$ is a Poisson process with intensity $\eta>0$, and $\{\xi_1,\xi_2,\ldots\}$ are iid (positive) random variables with common probability density function given by
\begin{equation*}
f_\xi(y) = \left( \sum_{i=1}^n a_i \alpha_i \mathrm e^{-\alpha_i y} \right) \ind_{\{y>0\}} ,
\end{equation*}
where $n$ is a positive integer, $0<\alpha_1<\alpha_2<\ldots<\alpha_n$ and $\sum_{i=1}^n a_i=1$, where $a_i>0$ for all $i=1,\ldots,n$. All of the aforementioned objects are mutually independent.

The Laplace exponent of $X$ is then clearly given by
$$
\psi(\lambda) = \drift \lambda + \frac12\sigma^2\lambda^2 + \eta \left( \sum_{i=1}^n \frac{a_i \alpha_i}{\lambda+\alpha_i} - 1 \right) , 
$$
for $\lambda>-\alpha_1$. In this case,
$$
\e \left[ X_1 \right] = \psi'(0+) = \drift - \eta \sum_{i=1}^n \frac{a_i}{\alpha_i} .
$$

For $q>0$ or $\psi'(0+) \neq 0$, one can write (see e.g.\ \cite{loeffenetal2012})
\begin{equation*}\label{partialfraction}
\frac{1}{\psi(\lambda)-q} = \sum_{i=1}^{n+2} \left( \psi' \left( \theta^{(q)}_i \right) \left( \lambda-\theta_i^{(q)} \right) \right)^{-1} ,
\end{equation*}
for $\lambda \in \mathbb R \backslash \left( \left\lbrace \theta_1^{(q)},\ldots,\theta_{n+2}^{(q)} \right\rbrace \cup \left\lbrace \alpha_1,\ldots,\alpha_n \right\rbrace \right)$, where $\theta_1^{(q)}> \theta_2^{(q)}> \ldots> \theta_{n+2}^{(q)}$ are the roots of $\lambda \mapsto \psi(\lambda)-q$ and are such that $\theta_1^{(q)}=\Phi(q)$ and $\theta_{n+2}^{(q)}  < -\alpha_n <  \theta_{n+1}^{(q)} < -\alpha_{n-1} < \theta_{n}^{(q)} \ldots <  -\alpha_1 < \theta_2^{(q)} < \theta_1^{(q)}$, since $\sigma$ is assumed here to be strictly positive.

In conclusion, by Laplace inversion, we have for $q>0$ or for $q=0$ and $\psi'(0) \neq 0$, that for $x \geq 0$
\begin{align*}
W^{(q)}(x) &= \sum_{i=1}^{n+2} \frac{ \mathrm e^{ \theta_i^{(q)} x} }{ \psi' \left( \theta_i^{(q)} \right) }, \\
Z^{(q)}(x) &=
\begin{cases}
q \sum_{i=1}^{n+2} \frac{ \mathrm e^{ \theta_i^{(q)} x} }{ \psi' \left( \theta_i^{(q)} \right) \theta_i^{(q)} } & \text{if $q>0$}, \\
1 & \text{if $q=0$}.
\end{cases} 
\end{align*}
Of course, $\mathbb{W}^{(q)}$ and $\mathbb{Z}^{(q)}$ will look just the same (we only need to change $\drift$ for $\drift-\alpha$ at the beginning of the above procedure). Therefore, computing derivatives and integrals of those scale functions, in particular expressions for $\wq$ and $\zq$, will be very easy, thanks to the exponential form of all those scale functions.

\begin{remark}
We could also have phase-type-distributed random variables, instead of a mixture of exponentials and still get scale functions that are the sum of exponential functions. See \cite{egamiyamazaki2010}.
\end{remark}

On the other hand  (when considering other examples for which  the $q$-scale function is not known in explicit form), there are good numerical methods for dealing with Laplace inversion  (cf. \cite{kuznetsovetal2011}*{Section 5} which deals specifically with Laplace inversion of the scale function.

Finally, for more details on spectrally negative L\'{e}vy processes and their use in ruin theory, the reader is referred to \cites{kyprianou2012, kyprianou2006}. For examples and numerical techniques related to the computation of scale functions, we suggest to look at \cite{kuznetsovetal2011}.

\section{Time spent in the red zone}

Now, we derive the joint Laplace transforms of
$$
\left( \kappa_a^- , \int_0^{\kappa_a^-} \ind_{\{U_s < b\}} \mathrm{d}s \right) \quad \text{and} \quad \left( \kappa_c^+ , \int_0^{\kappa_c^+} \ind_{\{U_s < b\}} \mathrm{d}s \right) ,
$$
where $a \leq x,b \leq c$, from which all subsequent results will be derived. Note that, for example, the random variable $\int_0^{\kappa_a^-} \ind_{\{U_s < b\}} \mathrm{d}s$ is the time spent by $U$ below $b$ (occupation time of the red zone) until level $a$ is crossed.

The structure of our main results is the same as in \cite{loeffenetal2012}, thanks to the Markovian property of $U$ (see Theorem 1). Note that we limit ourselves to the time spent in the red zone, as opposed to any interval; our methodology would apply to any finite interval at the cost of more complicated expressions, namely with extra convolution terms.
\begin{theorem}\label{T:stopaboveandbelow}
For  $a \leq x,b \leq c$ and for $p,q \geq 0$,
\begin{equation*}\label{E:stopabove}
\e_x \left[ \mathrm{e}^{- p \kappa_c^+ - q \int_0^{\kappa_c^+} \ind_{\{U_s < b\}} \mathrm{d}s } ; \kappa_c^+ < \kappa_a^- \right] =  \frac{w^{(p+q)}(x;a) - q \int_{b}^{x} \mathbb W^{(p)}(x-y) w^{(p+q)}(y;a) \mathrm{d}y}{w^{(p+q)}(c;a) - q \int_{b}^{c} \mathbb W^{(p)}(c-y) w^{(p+q)}(y;a) \mathrm{d}y} .
%
\end{equation*}
%
and
\begin{multline*}\label{E:stopbelow}
\e_x \left[ \mathrm{e}^{- p \kappa_a^- - q \int_0^{\kappa_a^-} \ind_{\{U_s < b\}} \mathrm{d}s } ; \kappa_a^- < \kappa_c^+ \right] \\
=  z^{(p+q)}(x;a) - q \int_{b}^{x} \mathbb W^{(p)}(x-y) z^{(p+q)}(y;a) \mathrm{d}y \\
+ \frac{z^{(p+q)}(c;a) - q \int_{b}^{c} \mathbb W^{(p)}(c-y) z^{(p+q)}(y;a) \mathrm{d}y}{w^{(p+q)}(c;a) - q \int_{b}^{c} \mathbb W^{(p)}(c-y) w^{(p+q)}(y;a) \mathrm{d}y} \\
\times \left( w^{(p+q)}(x;a) - q \int_{b}^{x} \mathbb W^{(p)}(x-y) w^{(p+q)}(y;a) \mathrm{d}y \right) .
\end{multline*}
\end{theorem}

These results are fundamental in developing a risk management toolkit based on occupation times, as we will see in the next sections. For example, we will derive several results needed to compute the probability of bankruptcy and the probability of Parisian ruin in this model with restructuring/refraction. They are extensions of those obtained in \cite{loeffenetal2012} for the case $\alpha=0$. Also, our results improve the results in \cite{kyprianouetal2012} because we are dealing with the case $p>0$ and we consider a general starting point $x$.

Before proving Theorem~\ref{T:stopaboveandbelow}, we will need the following technical lemma.
\subsection{Technical lemma}

Here is a generalization of Theorem 16 in \cite{kyprianouloeffen2010}, in the spirit of Lemma 2.1 in \cite{loeffenetal2012}. Recall that $W^{(q)}$ and $Z^{(q)}$ are the scale functions associated with $X$, while $\mathbb W^{(q)}$ and $\mathbb Z^{(q)}$ are those associated with $Y$.
\begin{lemma}\label{L:mainlemma}
For all $p,q \geq 0$ and $x,c$ such that $b \leq x \leq c$,
\begin{multline*}
\mathbf E_x \left[ \mathrm e^{-p\nu_b^-} W^{(q)}(Y_{\nu_b^-}) ; \nu_b^-<\nu_c^+ \right] = \wq(x;0) - (q-p)\int_b^x \mathbb W^{(p)}(x-y) \wq(y;0) \mathrm{d}y \\
- \frac{ \mathbb W^{(p)}(x-b)  }{ \mathbb W^{(p)}(c-b) } \left( \wq(c;0) - (q-p) \int_b^c \mathbb W^{(p)}(c-y) \wq(y;0) \mathrm{d}y \right) .
\end{multline*}
and
\begin{multline*}
\mathbf E_x \left[ \mathrm e^{-p\nu_b^-} Z^{(q)}(Y_{\nu_b^-}) ; \nu_b^-<\nu_c^+ \right] = \zq(x;0) - (q-p)\int_b^x \mathbb W^{(p)}(x-y) \zq(y;0) \mathrm{d}y \\
- \frac{ \mathbb W^{(p)}(x-b)  }{ \mathbb W^{(p)}(c-b) } \left( \zq(c;0) - (q-p) \int_b^c \mathbb W^{(p)}(c-y) \zq(y;0) \mathrm{d}y \right) .
\end{multline*}
\end{lemma}
\begin{proof}
Since $\{Y_t, t < \nu_b^-\}$ under $\mathbf{P}_x$ and $\{U_t, t < \kappa_b^-\}$ under $\p_x$ have the same law when $x \geq b$, we have
$$
\mathbf E_x \left[ \mathrm e^{-p\nu_b^-} W^{(q)}(Y_{\nu_b^-}) ; \nu_b^-<\nu_c^+ \right] = \mathbb E_x \left[ \mathrm e^{-q \kappa_b^-} W^{(q)}(U_{\kappa_b^-}) ; \kappa_b^-<\kappa_c^+ \right] .
$$
Consequently, according to Lemma 2.1 in \cite{loeffenetal2012}, it suffices to show that
$$
\mathbb E_x \left[ \mathrm e^{-q \kappa_b^-} \wq (U_{\kappa_b^-};0) ; \kappa_b^-<\kappa_c^+ \right] = \wq(x;0) - \frac{\mathbb W^{(q)}(x-b)}{\mathbb W^{(q)}(c-b)} \wq(c;0) .
$$
The latter is easily obtained using the strong Markov property of $U$ (see Theorem 1) and the solution to the two-sided exit problem for $U$ (see Theorem~\ref{T:twosidedforU}): indeed, for $b \leq x \leq c$, we can write
\begin{align*}
\frac{\wq(x;0)}{\wq(c;0)} &= \e_x \left[ \mathrm{e}^{-q\kappa_c^+} ; \kappa_c^+<\kappa_0^- \right] \\
&= \e_x \left[ \mathrm{e}^{-q\kappa_c^+} ; \kappa_c^+<\kappa_b^- \right] + \e_x \left[ \mathrm{e}^{-q \kappa_c^+} ; \kappa_b^-<\kappa_c^+<\kappa_0^- \right] \\
&=  \frac{\mathbb{W}^{(q)}(x-b)}{\mathbb{W}^{(q)}(c-b)} + \mathbf{E}_x \left[ \mathrm{e}^{-q \nu_b^-} \e_{Y_{\nu_b^-}} \left[ \mathrm{e}^{-q \tau_b^+} ; \tau_b^+<\tau_0^- \right] ; \nu_b^- < \nu_c^+ \right] \\
&=  \frac{\mathbb{W}^{(q)}(x-b)}{\mathbb{W}^{(q)}(c-b)} + \e_b \left[ \mathrm{e}^{-q\kappa_c^+} ; \kappa_c^+<\kappa_0^- \right] \mathbf{E}_x \left[ \mathrm{e}^{-q \nu_b^-} \frac{W^{(q)}(Y_{\nu_b^-})}{W^{(q)}(b)} ; \nu_b^- < \nu_c^+ \right] ,
\end{align*}
where we used again that $\{Y_t, t < \nu_b^-\}$ under $\mathbf{P}_x$ and $\{U_t, t < \kappa_b^-\}$ under $\p_x$ have the same law when $x \geq b$, but also that $\{X_t, t < \tau_b^+\}$ and $\{U_t, t < \kappa_b^+\}$ have the same law under $\p_x$ when $x \leq b$. From the definition of $\wq(\cdot;0)$, we know that $W^{(q)}(\cdot)$ and $\wq(\cdot;0)$ coincide on $(-\infty,b]$, so then we have that $W^{(q)}(Y_{\nu_b^-}) = \wq (Y_{\nu_b^-};0)$.

As in \cite{loeffenetal2012}, for $q \geq 0$, we can define $\mathcal V^{(q)}_b (Y)$ to be the function space associated with the SNLP $Y$ consisting of positive and measurable functions $v^{(q)}(x)$, $x\in(-\infty,\infty)$, that satisfy:
\begin{equation}\label{mgproperty}
\mathbf E_x \left[ \mathrm e^{-q\nu_b^-} v^{(q)}(Y_{\nu_b^-}) ; \nu_b^-<\nu_c^+ \right] = v^{(q)}(x) - \frac{\mathbb W^{(q)}(x-b)}{\mathbb W^{(q)}(c-b)} v^{(q)}(c) ,
\end{equation}
for all $x,c$ such that $b \leq x \leq c$. From the above calculations, it follows that $\wq(\cdot;0)$ satisfies Property~\eqref{mgproperty} and thus $\wq(\cdot;0) \in \mathcal V^{(q)}_b (Y)$, for all $q,b\geq0$.
The result follows from Lemma 2.1 in \cite{loeffenetal2012}.

Similarly, one can prove that $\zq(\cdot;0) \in \mathcal V^{(q)}_b (Y)$, for all $q,b \geq 0$.
\end{proof}

\subsection{Proof of Theorem~\ref{T:stopaboveandbelow}}

Fix $a \leq b \leq c$, for $p,q \geq 0$. For $x \in [a,c]$, define
$$
v(x) = \e_x \left[ \mathrm{e}^{- p \kappa_c^+ - q \int_0^{\kappa_c^+} \ind_{\{U_s < b\}} \mathrm{d}s } ; \kappa_c^+ < \kappa_a^- \right] .
$$
Using the strong Markov property of $X$, the fact that $X$ is skip-free upward and \eqref{E:exitabove}, we can write, for $a \leq x < b$,
\begin{equation*}
v(x) = v(b) \e_x \left[  \mathrm{e}^{-(p+q) \kappa_b^+}; \kappa_b^+<\kappa_a^- \right] = v(b) \e_x \left[  \mathrm{e}^{-(p+q) \tau_b^+}; \tau_b^+<\tau_a^- \right] = v(b) \frac{W^{(p+q)}(x-a)}{W^{(p+q)}(b-a)}.
\end{equation*}
Similarly, for $b \leq x \leq c$, we have
\begin{align}\label{E:above}
v(x) &=  \e_x \left[ \mathrm{e}^{-p \kappa_c^+} ; \kappa_c^+ < \kappa_b^- \right] + \e_x \left[ \mathrm{e}^{-p \kappa_b^-} v \left( U_{\kappa_b^-} \right) ; \kappa_b^- < \kappa_c^+ \right] \notag \\
&=  \frac{\mathbb{W}^{(p)}(x-b)}{\mathbb{W}^{(p)}(c-b)} + \frac{v(b)}{W^{(p+q)}(b-a)} \e_x \left[ \mathrm{e}^{-p \kappa_b^-} W^{(p+q)} \left( U_{\kappa_b^-} - a \right) ; \kappa_b^- < \kappa_c^+ \right] .
\end{align}

We now assume that $X$ has paths of bounded variation. In this case, we have $\mathbb{W}^{(p)}(0) \neq 0$ and thus setting $x=b$ in \eqref{E:above} yields
\begin{equation}\label{E:atb}
v(b) = \frac{\mathbb{W}^{(p)}(0)/\mathbb{W}^{(p)}(c-b)}{1-\frac{1}{W^{(p+q)}(b-a)} \e_b \left[ \mathrm{e}^{-p \kappa_b^-} W^{(p+q)} \left( U_{\kappa_b^-} - a \right) ; \kappa_b^- < \kappa_c^+ \right]} .
\end{equation}
Since $\{Y_t, t < \nu_b^-\}$ under $\mathbf{P}_x$ and $\{U_t, t < \kappa_b^-\}$ under $\p_x$ have the same law when $x \geq b$, and by spatial homogeneity of $Y$, we have
\begin{multline*}
\e_b \left[ \mathrm{e}^{-p \kappa_b^-} W^{(p+q)} \left( U_{\kappa_b^-} - a \right) ; \kappa_b^- < \kappa_c^+ \right] \\
= \mathbf{E}_b \left[ \mathrm{e}^{-p \nu_b^-} W^{(p+q)} \left( Y_{\nu_b^-} - a \right) ; \nu_b^- < \nu_c^+ \right] \\
= \mathbf{E}_{b-a} \left[ \mathrm{e}^{-p \nu_{b-a}^-} W^{(p+q)} \left( Y_{\nu_{b-a}^-} \right) ; \nu_{b-a}^- < \nu_{c-a}^+ \right] .
\end{multline*}
Using Lemma~\ref{L:mainlemma}, we then have that
\begin{multline*}
\e_b \left[ \mathrm{e}^{-p \kappa_b^-} W^{(p+q)} \left( U_{\kappa_b^-} - a \right) ; \kappa_b^- < \kappa_c^+ \right] \\
= w^{(p+q)}(b-a;b-a) - \frac{\mathbb W^{(p)}(0)}{\mathbb W^{(p)}(c-b)} \left( w^{(p+q)}(c;a) - q \int_{b}^{c} \mathbb W^{(p)}(c-y) w^{(p+q)}(y;a) \mathrm{d}y \right) ,
\end{multline*}
where $w^{(p+q)}(b-a;b-a) = W^{(p+q)}(b-a)$. Plugging this into \eqref{E:atb} yields
$$
v(b) = \frac{W^{(p+q)}(b-a)}{w^{(p+q)}(c;a) - q \int_{b}^{c} \mathbb W^{(p)}(c-y) w^{(p+q)}(y;a) \mathrm{d}y} .
$$
Plugging now the value of $v(b)$ just obtained into \eqref{E:above} and using again Lemma~\ref{L:mainlemma} yields, for $a \leq x \leq c$,
\begin{equation*}
v(x) =  \frac{w^{(p+q)}(x;a) - q \int_{b}^{x} \mathbb W^{(p)}(x-y) w^{(p+q)}(y;a) \mathrm{d}y}{w^{(p+q)}(c;a) - q \int_{b}^{c} \mathbb W^{(p)}(c-y) w^{(p+q)}(y;a) \mathrm{d}y} .
\end{equation*}

The case where $X$ has paths of unbounded variation follows using the same approximating procedure as in \cite{loeffenetal2012} (see also \cite{kyprianouetal2012}).

The proof of the second part of the Theorem is similar. For sake of brevity, the details are left to the reader.

\section{Probability of bankruptcy}

We now apply our main result to compute the probability of bankruptcy. As mentioned previously, we consider level $b$ as a solvency capital requirement level and the interval $(-\infty,b)$ as the \textit{red zone}. We choose the following definition for bankruptcy: if $U$ spends too much time in the red zone or if $U$ drops too deep, then bankruptcy is declared. To be more precise, for $q>0$, define the function $\omega \colon \mathbb R \to [0,\infty)$ by
$$
\omega(x) =
\begin{cases}
0 & \text{if $x \geq b$}, \\
q & \text{if $0 \leq x < b$}, \\
\infty & \text{if $x<0$}.
\end{cases}
$$
and the corresponding bankruptcy time $\rho_\omega$ by
\begin{equation*}
 \rho_\omega = \inf \left\lbrace t>0 \colon \int_0^t \omega (U_s) \mathrm{d}s > \mathbf e_1 \right\rbrace ,
\end{equation*}
where $\mathbf e_1$ is an independent exponentially distributed random variable with rate $1$. Therefore, bankruptcy occurs at rate $q$ when $U$ is between $0$ and $b$, and bankruptcy occurs immediately if $U$ falls below level $0$. The choice of $0$ as the ultimate acceptable surplus level is arbitrary and not restrictive.

\begin{remark}
This definition of bankruptcy is borrowed from Omega models, in which the function $\omega$ is called the rate function. Typically, the rate function is chosen to be a decreasing function equalling zero above the critical level ($b$ in our case) so that bankruptcy does not occur in this situation. This family of models was introduced in \cite{albrecheretal2011b} and further investigated in \cite{gerberetal2012} for Brownian motion with drift, in \cite{loeffenetal2012} for spectrally negative L\'evy processes, and in \cite{albrecherlautscham2013} for compound Poisson processes and more general bankruptcy rate function. All these papers deal with the case $\alpha=0$.
%
\end{remark}


Suppose that the positive loading condition holds, which in our model means that $\e[X_1]>\alpha$. This implies that bankruptcy does not happen almost surely (see Equation~\eqref{E:refractedruinproba}) and, if it does occur, we have that either it occurs while the surplus is between $0$ and $b$ or it occurs due to the surplus process dropping below the level $0$. Mathematically, for any initial surplus $x \in \mathbb R$, we clearly have the following relationship:
$$
1 = \mathbb{P}_x \left( 0 \leq U_{\rho_\omega} < b, \rho_\omega<\infty \right)  + \mathbb{P}_x \left( U_{\rho_\omega}< 0, \rho_\omega<\infty \right) + \mathbb{P}_x \left( \rho_\omega=\infty \right) .
$$
The probability that bankruptcy occurs while the surplus is between $0$ and $b$ is given by
$$
\mathbb{P}_x \left( 0 \leq U_{\rho_\omega} < b, \rho_\omega<\infty \right) = \mathbb P_x \left( \int_0^{\kappa_0^-} \omega (U_s) \mathrm{d}s > \mathbf e_1 \right) = 1-\mathbb{E}_x \left[ \mathrm{e}^{-q \int_0^{\kappa^-_0} \ind_{\{U_s < b\}} \mathrm{d}s} \right] .
$$
Similarly, the probability that bankruptcy occurs due to the surplus process dropping below $0$ is given by
$$
\mathbb{P}_x \left( U_{\rho_\omega} < 0, \rho_\omega<\infty \right) = \mathbb P_x \left( \int_0^{\kappa_0^-} \omega (U_s) \mathrm{d}s \leq \mathbf e_1, \kappa_0^-<\infty \right) = \mathbb{E}_x \left[ \mathrm{e}^{-q \int_0^{\kappa^-_0} \ind_{\{U_s < b\}} \mathrm{d}s} ; \kappa^-_0<\infty \right] .
$$

In conclusion, the answer is included in the following corollary:
\begin{corollary} Assume the net profit condition $\e[X_1]>\alpha$ is verified.

\begin{itemize}

\item[(i)] For $x,b,q \geq 0$,
\begin{multline*}
\e_x \left[ \mathrm{e}^{- q \int_0^{\kappa_0^-} \ind_{\{U_s < b\}} \mathrm{d}s } ; \kappa_0^- < \infty \right] \\
=  z^{(q)}(x;0) - q \int_{b}^{x} \mathbb W(x-y) z^{(q)}(y;0) \mathrm{d}y \\
+ \frac{\left( \e[X_1] - \alpha \right) + q \int_0^b \left( \mathbb{Z}^{(q)}(y) - \alpha W^{(q)}(y) \mathbb{Z}^{(q)}(b-y) \right) \mathrm{d}y}{Z^{(q)}(b) - \alpha W^{(q)}(b)} \\
\times \left( w^{(q)}(x;0) - q \int_{b}^{x} \mathbb W(x-y) w^{(q)}(y;0) \mathrm{d}y \right) .
\end{multline*}

\item[(ii)] For $x,b,q \geq 0$,
\begin{equation*}
\mathbb{E}_x \left[\mathrm{e}^{-q\int_0^{\kappa^-_{0}} \ind_{\{U_s < b\}} \mathrm{d}s} ; \kappa_0^- = \infty \right] = (\e[X_1]-\alpha) \frac{w^{(q)}(x;0) - q \int_b^x \mathbb{W}(x-y) w^{(q)}(y;0) \mathrm{d}y}{Z^{(q)}(b) - \alpha W^{(q)}(b)}
\end{equation*}
\end{itemize}
\end{corollary}
\begin{proof}
For part (i), we clearly have that
$$
\mathbb{E}_x \left[\mathrm{e}^{-q\int_0^{\kappa^-_{0}} \ind_{\{U_s < b\}} \mathrm{d}s} ; \kappa_0^- < \infty \right] = \lim_{c \to \infty} \mathbb{E}_x \left[\mathrm{e}^{-q \int_0^{\kappa^-_{0}} \ind_{\{U_s < b\}} \mathrm{d}s} ; \kappa_0^- < \kappa_c^+ \right] .
$$
From Theorem~\ref{T:stopaboveandbelow}, we know that
\begin{multline*}
\e_x \left[ \mathrm{e}^{- q \int_0^{\kappa_0^-} \ind_{\{U_s < b\}} \mathrm{d}s } ; \kappa_0^- < \kappa_c^+ \right] \\
=  z^{(q)}(x;0) - q \int_{b}^{x} \mathbb W(x-y) z^{(q)}(y;0) \mathrm{d}y \\
+ \frac{z^{(q)}(c;0) - q \int_{b}^{c} \mathbb W(c-y) z^{(q)}(y;0) \mathrm{d}y}{w^{(q)}(c;0) - q \int_{b}^{c} \mathbb W(c-y) w^{(q)}(y;0) \mathrm{d}y} \\
\times \left( w^{(q)}(x;0) - q \int_{b}^{x} \mathbb W(x-y) w^{(q)}(y;0) \mathrm{d}y \right) .
\end{multline*}
It can be shown that, for $c>b$,
\begin{multline*}
w^{(q)}(c;0) - q \int_{b}^{c} \mathbb W(c-y) w^{(q)}(y;0) \mathrm{d}y \\
= \left( 1-\alpha W^{(q)}(0) \right) \left\lbrace \mathbb W(c) + q \int_0^b \mathbb W(c-y) \mathbb W^{(q)}(y) \mathrm{d}y \right\rbrace \\
- \alpha \int_0^b W^{(q) \prime} (z) \left\lbrace \mathbb W(c-z) + q \int_0^{b-z} \mathbb W(c-z-y) \mathbb W^{(q)}(y) \mathrm{d}y \right\rbrace \mathrm{d}z .
\end{multline*}
To prove this last identity, we used Equations~\eqref{E:convolution} and~\eqref{E:rep_wq}; the details are left to the reader. Therefore, since it is assumed that $\alpha < \e [X_1]$ we have $\lim_{c \to \infty} \mathbb{W}(c) = \left( \psi'(0+) - \alpha \right)^{-1} = \left( \e[X_1] - \alpha \right)^{-1}$, and then we get (using the monotone convergence theorem)
\begin{multline*}
\lim_{c \to \infty} \frac{w^{(q)}(c;0) - q \int_{b}^{c} \mathbb W(c-y) w^{(q)}(y;0) \mathrm{d}y}{\mathbb{W}(c)}\\
= \left( 1-\alpha W^{(q)}(0) \right) \mathbb{Z}^{(q)}(b) - \alpha \int_0^b W^{(q) \prime} (z) \mathbb{Z}^{(q)}(b-z) \mathrm{d}z .
\end{multline*}
Integrating by parts (or taking Laplace transforms on both sides) yields
$$
Z^{(q)}(x) - \alpha W^{(q)}(x) = \left( 1-\alpha W^{(q)}(0) \right) \mathbb{Z}^{(q)}(x) - \alpha \int_0^x W^{(q) \prime} (y) \mathbb{Z}^{(q)}(x-y) \mathrm{d}y .
$$
We can also show that, for $c>b$,
\begin{multline*}
z^{(q)}(c;0) - q \int_{b}^{c} \mathbb W(c-y) z^{(q)}(y;0) \mathrm{d}y \\
= 1 + q \int_0^b \mathbb W(c-y) \mathbb Z^{(q)}(y) \mathrm{d}y \\
- \alpha q \int_0^b W^{(q)} (z) \left\lbrace \mathbb W(c-z) + q \int_0^{b-z} \mathbb W(c-z-y) \mathbb W^{(q)}(y) \mathrm{d}y \right\rbrace \mathrm{d}z .
\end{multline*}
Again, the details are left to the reader. Then, as above, we get
\begin{multline*}
\lim_{c \to \infty} \frac{z^{(q)}(c;0) - q \int_{b}^{c} \mathbb W(c-y) z^{(q)}(y;0) \mathrm{d}y}{\mathbb{W}(c)}\\
= \left( \e[X_1] - \alpha \right) + q \int_0^b \left( \mathbb{Z}^{(q)}(y) - \alpha W^{(q)}(y) \mathbb{Z}^{(q)}(b-y) \right) \mathrm{d}y ,
\end{multline*}
and the result follows.

For part (ii), we also have that
\begin{equation*}
\mathbb{E}_x \left[\mathrm{e}^{-q\int_0^{\kappa^-_{0}} \ind_{\{U_s < b\}} \mathrm{d}s} ; \kappa_0^- = \infty \right] = \lim_{c \to \infty} \mathbb{E}_x \left[\mathrm{e}^{-q\int_0^{\kappa^+_c} \ind_{\{U_s < b\}} \mathrm{d}s} ; \kappa^+_c < \kappa_0^- \right] .
\end{equation*}
From Theorem~\ref{T:stopaboveandbelow}, we know that
$$
\e_x \left[ \mathrm{e}^{- q \int_0^{\kappa_c^+} \ind_{\{U_s < b\}} \mathrm{d}s } ; \kappa_c^+ < \kappa_0^- \right] =  \frac{w^{(q)}(x;0) - q \int_{b}^{x} \mathbb W(x-y) w^{(q)}(y;0) \mathrm{d}y}{w^{(q)}(c;0) - q \int_{b}^{c} \mathbb W(c-y) w^{(q)}(y;0) \mathrm{d}y}
$$
and, by the above, the result follows.
\end{proof}

\section{Probability of Parisian ruin}

In \cite{landriaultetal2013}, a definition of Parisian ruin is proposed. For this definition of ruin, each excursion of the surplus process $U$ below the critical level $b$ is accompanied by an independent copy of an independent (of $U$) random variable. It is called the implementation clock. If the duration of a given excursion below $b$ is less than its associated implementation clock, then ruin does not occur. Ruin occurs at the first time $\tau_q$ that an implementation clock rings before the end of its corresponding excursion below $b$.

It can be shown that, if the implementation clock is exponentially distributed with rate $q$, then the probability of Parisian ruin is given by
$$
\p_x \left(  \tau_q < \infty \right)  = 1 - \e_x \left[ \mathrm{e}^{- q \int_0^{\infty} \mathbb{I}_{\{U_s < b\}} \mathrm{d}s } \right] ,
$$
when the net profit condition is verified.

\begin{corollary}
Assume the net profit condition $\e[X_1]>\alpha$ is verified.

\begin{itemize}

\item[(i)] For  $x,b \leq c$ and $q \geq 0$,
\begin{equation*}
\e_x \left[ \mathrm{e}^{- q \int_0^{\kappa_c^+} \ind_{\{U_s < b\}} \mathrm{d}s } ; \kappa_c^+ < \infty \right] =  \frac{\mathrm{e}^{\Phi(q) (x-b)} \left( 1 - \left( q-\alpha \Phi(q) \right) \int_0^{x-b} \mathrm{e}^{-\Phi(q) y} \mathbb{W}(y) \mathrm{d}y \right) }{\mathrm{e}^{\Phi(q) (c-b)} \left( 1 - \left( q-\alpha \Phi(q) \right) \int_0^{c-b} \mathrm{e}^{-\Phi(q) y} \mathbb{W}(y) \mathrm{d}y \right) } .
\end{equation*}

\item[(ii)] For $q \geq 0$ and $x \in \reals$ we have
\begin{multline*}
\e_x \left[ \mathrm{e}^{-q \int_0^\infty \ind_{\{U_s < b\}} \mathrm{d}s} \right] \\
= \left( \frac{ \left( \e[X_1]-\alpha \right) \Phi(q)}{q - \alpha \Phi(q)} \right) \mathrm{e}^{\Phi(q) (x-b)} \left( 1 - \left( q-\alpha \Phi(q) \right) \int_0^{x-b} \mathrm{e}^{-\Phi(q) y} \mathbb{W}(y) \mathrm{d}y \right) .
\end{multline*}
Moreover, we get the distribution
\begin{multline*}
\p_x \left( \int_0^\infty \ind_{\{U_s < b\}} \mathrm{d}s \in \mathrm{d}r \right) \\
= \left( \e[X_1]-\alpha \right) \left\lbrace \mathbb{W}(x-b) \delta_0 (\mathrm{d}r) + \int_0^\infty \frac{y}{r} \mathbb{W}^\prime (y+x-b) \p \left( X_r \in \mathrm{d}y \right) \mathrm{d}r \right\rbrace .
\end{multline*}
\end{itemize}

\end{corollary}
\begin{proof}
For part (i), we clearly have that
$$
\mathbb{E}_x \left[\mathrm{e}^{-q\int_0^{\kappa^+_c} \ind_{\{U_s < b\}} \mathrm{d}s} ; \kappa_c^+ < \infty \right] = \lim_{m \to \infty} \mathbb{E}_x \left[\mathrm{e}^{-q \int_0^{\kappa^+_c} \ind_{\{U_s < b\}} \mathrm{d}s} ; \kappa_c^+ < \kappa_{-m}^- \right] .
$$
From Theorem~\ref{T:stopaboveandbelow}, we have that
\begin{equation*}
\mathbb{E}_x \left[\mathrm{e}^{-q \int_0^{\kappa^+_c} \ind_{\{U_s < b\}} \mathrm{d}s} ; \kappa_c^+ < \kappa_{-m}^- \right]\\
=  \frac{w^{(q)}(x;-m) - q \int_{b}^{x} \mathbb W(x-y) w^{(q)}(y;-m) \mathrm{d}y}{w^{(q)}(c;-m) - q \int_{b}^{c} \mathbb W(c-y) w^{(q)}(y;-m) \mathrm{d}y} .
\end{equation*}
Since, for any $x$,
$$
\lim_{m \to \infty} \frac{W^{(q)}(x+m)}{W^{(q)}(m)} = \mathrm{e}^{\Phi(q) x}
$$
and
$$
\lim_{m \to \infty} \frac{W^{(q) \prime}(x+m)}{W^{(q)}(m)} = \Phi(q) \mathrm{e}^{\Phi(q) x} ,
$$
we get
$$
\lim_{m \to \infty} \frac{w^{(q)}(x;-m)}{W^{(q)}(m)} = \mathrm{e}^{\Phi(q) x} + \alpha \Phi(q) \int_b^x \mathrm{e}^{\Phi(q) y} \mathbb{W}^{(q)}(x-y) \mathrm{d}y .
$$
Consequently,
\begin{multline*}
\lim_{m \to \infty} \frac{w^{(q)}(x;-m) - q \int_{b}^{x} \mathbb W(x-y) w^{(q)}(y;-m) \mathrm{d}y}{W^{(q)}(m)} \\
= \mathrm{e}^{\Phi(q) x} + \alpha \Phi(q) \int_b^x \mathrm{e}^{\Phi(q) y} \mathbb{W}^{(q)}(x-y) \mathrm{d}y \\
- q \int_b^x \mathbb{W}(x-y) \left( \mathrm{e}^{\Phi(q) y} + \alpha \Phi(q) \int_b^y \mathrm{e}^{\Phi(q) z} \mathbb{W}^{(q)}(y-z) \mathrm{d}z \right) \mathrm{d}y .
\end{multline*}
Using Equation~\eqref{E:convolution} among other arguments, it can be shown that
\begin{multline*}
\int_b^x \mathbb{W}(x-y) \left( \mathrm{e}^{\Phi(q) y} + \alpha \Phi(q) \int_b^y \mathrm{e}^{\Phi(q) z} \mathbb{W}^{(q)}(y-z) \mathrm{d}z \right) \mathrm{d}y \\
= \left( 1 - \frac{\alpha \Phi(q)}{q} \right) \int_b^x \mathrm{e}^{\Phi(q) y} \mathbb{W}(x-y) \mathrm{d}y + \left( \frac{\alpha \Phi(q)}{q} \right) \int_b^x \mathrm{e}^{\Phi(q) y} \mathbb{W}^{(q)}(x-y) \mathrm{d}y ,
\end{multline*}
and the result follows.

For part (ii), since the net profit condition is assumed, we have that $\kappa_c^+ < \infty$ almost surely, for any $c$, and we have that
$$
\lim_{x \to \infty} \mathbb{W}(x) = \left( \e[X_1]-\alpha \right)^{-1} .
$$
Therefore, we can write
\begin{equation*}
\e_x \left[ \mathrm{e}^{-q \int_0^\infty \ind_{\{U_s < b\}} \mathrm{d}s} \right] = \lim_{c \to \infty} \mathbb{E}_x \left[\mathrm{e}^{-q\int_0^{\kappa^+_c} \ind_{\{U_s < b\}} \mathrm{d}s} ; \kappa_c^+ < \infty \right] .
\end{equation*}
From the result in part (i), we then have that
$$
\e_x \left[ \mathrm{e}^{-q \int_0^\infty \ind_{\{U_s < b\}} \mathrm{d}s} \right] = \frac{\mathrm{e}^{\Phi(q) (x-b)} \left( 1 - \left( q-\alpha \Phi(q) \right) \int_0^{x-b} \mathrm{e}^{-\Phi(q) y} \mathbb{W}(y) \mathrm{d}y \right)}{\lim_{c \to \infty} \mathrm{e}^{\Phi(q) (c-b)} \left( 1 - \left( q-\alpha \Phi(q) \right) \int_0^{c-b} \mathrm{e}^{-\Phi(q) y} \mathbb{W}(y) \mathrm{d}y \right) } ,
$$
if the limit exists. In fact, for $c > b$, we can show that
$$
\mathrm{e}^{\Phi(q) c} \left( 1 - \left( q-\alpha \Phi(q) \right) \int_0^{c-b} \mathrm{e}^{-\Phi(q) y} \mathbb{W}(y) \mathrm{d}y \right) = \left( q-\alpha \Phi(q) \right) \int_{-b}^\infty \mathrm{e}^{-\Phi(q) y} \mathbb{W}(y+c) \mathrm{d}y .
$$
Since
$$
\lim_{c \to \infty} \left( q-\alpha \Phi(q) \right) \frac{\int_{-b}^\infty \mathrm{e}^{-\Phi(q) y} \mathbb{W}(y+c) \mathrm{d}y}{\mathbb{W}(c)} = \left( \frac{q-\alpha \Phi(q)}{\Phi(q)} \right) \mathrm{e}^{\Phi(q) b} ,
$$
the result follows.

Now, to extract the probability distribution from the Laplace transform, we first modify the expression just obtained. Set
$$
v(x) = \e_x \left[ \mathrm{e}^{-q \int_0^\infty \ind_{\{U_s < b\}} \mathrm{d}s} \right] .
$$
After a few manipulations (definition of the $0$-scale function and integration by parts), one can write
$$
v(x) = \left( \e[X_1]-\alpha \right) \left[ \mathbb{W}(x-b) + \int_0^\infty \mathrm{e}^{-\Phi(q) y} \mathbb{W}^\prime(y+x-b) \mathrm{d}y \right] .
$$
Therefore,
$$
v(x) = \left( \e[X_1]-\alpha \right) \left[ \mathbb{W}(x-b) + \int_0^\infty \left\lbrace \int_0^\infty \mathrm{e}^{-q s} \p \left( \tau_y^+ \in \mathrm{d}s \right) \right\rbrace \mathbb{W}^\prime(y+x-b) \mathrm{d}y \right] .
$$
Note that, by Kendall's identity, we have on $(0,\infty) \times (0,\infty)$
$$
\mathrm{d}y \p \left( \tau_y^+ \in \mathrm{d}r \right) = \frac{y}{r} \p \left( X_r \in \mathrm{d}y \right) \mathrm{d}r ,
$$
and the result follows.

\end{proof}

Note that if we set $\alpha=0$ in (i) of the last corollary, we recover Corollary 2(ii) in \cite{loeffenetal2012}. Note also that (ii) of the last corollary is a slight improvement over Corollary 2 in \cite{kyprianouetal2012}: any initial level $X_0 = x \in \reals$ is considered; the proof is also different. Note finally that we have obtained a different expression for the density of the total time spent by $U$ below level $b$.

\section{Acknowledgements}

Funding in support of this work was provided by the Natural Sciences and Engineering Research Council of Canada (NSERC), the Fonds de recherche du Qu\'ebec - Nature et technologies (FRQNT) and the Insti\-tut de finance math\'ematique de Montr\'eal (IFM2).

%
%
\bibliographystyle{abbrv}
\bibliography{refracted_occ-time_19juin2013}

\end{document}